\documentclass{article}

\usepackage[a4paper, total={6in, 8in}]{geometry}
\usepackage{emptypage}
\usepackage[utf8]{inputenc}
\usepackage[T1]{fontenc}
\usepackage{lmodern}
\usepackage[english]{babel}
\usepackage{amsmath, amssymb}
\usepackage{amsthm}
\usepackage{longtable}
\usepackage{placeins}
\usepackage[font=small,labelfont=bf]{caption}

\usepackage{enumerate}
\usepackage{xcolor}
\RequirePackage[colorlinks,allcolors=blue]{hyperref}
\RequirePackage{graphicx}
\usepackage[nottoc]{tocbibind}

\usepackage{titlesec}
\titleformat{\section}{\large\bfseries}{\thesection}{1em}{}

\usepackage[round]{natbib} 
\let\oldcite\cite
\renewcommand{\cite}[1]{\mbox{\oldcite{#1}}}

\newtheorem{theorem}{Theorem}

\theoremstyle{definition}

\newtheorem*{remark*}{Remark}

\title{Asymptotic Normality of $U$-Statistics is Equivalent to Convergence in the Wasserstein Distance}
\date{\today}
\author{Marius Kroll\footnote{Faculty of Mathematics, Ruhr-Universität Bochum, \href{mailto:marius.kroll@rub.de}{marius.kroll@rub.de}}}
\begin{document}
\maketitle

\begin{abstract}
  We prove the claim in the title under mild conditions which are usually satisfied when trying to establish asymptotic normality. We assume strictly stationary and absolutely regular data.\\[2mm]

\end{abstract}

\medskip
\section{Introduction}
Consider a $U$-statistic
$$
  U_h := U_h(X_1, \ldots, X_n) := {n \choose m}^{-1} \sum_{1 \leq i_1 < \ldots < i_m \leq n} h(X_{i_1}, \ldots, X_{i_m}),
$$
where the observations $X_1, \ldots, X_n$ take values in some separable metric space $\mathcal{X}$ and $h : \mathcal{X}^m \to \mathbb{R}$ is a kernel function that we assume without loss of generality to be centred and symmetric in its arguments. The class of $U$-statistics is a large one; it includes the sample mean and variance, estimators for the $k$-th moments, Gini's mean difference, pointwise evaluations of the cumulative distribution function and the distance covariance \citep{szekely_rizzo_bakirov:2007}. For a general introduction and more examples of $U$-statistics, see \cite{serfling:approximation_theorems} or \cite{vandervaart:asymptotic_statistics}.

Suppose you want to establish asymptotic normality of $U_h$, i.e.\@ you want to show that there is some $\sigma^2 > 0$ such that $\sqrt{n} U_h$ converges weakly to $\mathcal{N}\left(0, \sigma^2\right)$. Possible motivations for this are the construction of confidence intervals, finding critical values for hypothesis testing or completing an intermediate step in a larger proof (e.g.\@ bootstrap consistency). Following classical approaches in $U$-statistic theory, you consider the Hoeffding decomposition \citep{hoeffding:1948}, i.e.\@ the representation
\begin{equation}
  \label{eq:hoeffding_zerlegung}
  U_h = \sum_{i=0}^m {m \choose i} U_n^{(i)}(h;\xi),
\end{equation}
where $\xi := \mathcal{L}(X_1)$ and $U_n^{(i)}(h;\xi)$ is the $U$-statistic with kernel function
$$
  h_i(x_1, \ldots, x_i ; \xi) := \sum_{k=0}^i {i \choose k} (-1)^{i-k} \,\textbf{h}_k(x_1, \ldots, x_k; \xi),
$$
and
$$
  \textbf{h}_k(x_1, \ldots, x_k;\xi) := \int h(x_1, \ldots, x_m) ~\mathrm{d}\xi^{m-k}(x_{k+1}, \ldots, x_m).
$$
Because $h$ is centred, the constant term $U_n^{(0)}$ is equal to $0$ for all $n$. If you can show that $\sqrt{n}\,U_n^{(1)}(h;\xi)$ -- which is just a partial sum with normalisation $1/\sqrt{n}$ -- is asymptotically normal and that the remaining terms in Eq.\@ \eqref{eq:hoeffding_zerlegung} are $o_\mathbb{P}(1/\sqrt{n})$, you have successfully proven asymptotic normality of $U_h$. To do either of these things, you probably assume some moment conditions on your kernel $h$: second moments if your data are independent and identically distributed (i.i.d.), $p$-moments for some $p > 2$ if they are strictly stationary and mixing. We posit that these assumptions are already sufficient to establish the asymptotic normality not only in a weak sense, but also in the Wasserstein metric $d_2$, defined by
$$
  d_2^2(X,Y) := \inf \mathbb{E}\left[(X' - Y')^2\right],
$$
where the infimum is taken over all random vectors $(X', Y')$ such that the distributions of $X'$ and $Y'$ are equal to those of $X$ and $Y$, respectively. Convergence in $d_2$ is stronger than weak convergence alone as it implies certain moment convergences as well. More precisely, $d_2(Y_n, Y) \to 0$ if and only if $Y_n$ converges weakly to $Y$ and $\mathbb{E}[Y_n^2] \to \mathbb{E}[Y^2]$ \citep[Theorem 6.9 in][]{villani:optimal_transport}.

\section{Main Result}
\begin{theorem}
  \label{thm:hauptresultat}
  Let $(X_k)_{k \in \mathbb{N}}$ be a process with values in a separable metric space $\mathcal{X}$ and $h: \mathcal{X}^m \to \mathbb{R}$ be a centred and symmetric kernel. Suppose that one of the following two assumptions holds:
  \begin{enumerate}
    \item $(X_k)_{k \in \mathbb{N}}$ is i.i.d. and $\|h(X_1, \ldots, X_m)\|_{L_2} < \infty$.
    \item $(X_k)_{k \in \mathbb{N}}$ is strictly stationary and absolutely regular and $\|h(X_1, \ldots, X_m)\|_{L_p} < \infty$ for some $p > 2$. Furthermore, the mixing coefficients of $(X_k)_{k \in \mathbb{N}}$ satisfy the growth rate $\beta(n) = \mathcal{O}\left(n^{-r}\right)$ for some $r > mp/(p-2)$.
  \end{enumerate}
  Then $\sqrt{n}U_h$ converges in distribution to $\mathcal{N}\left(0, \sigma^2\right)$ for some $\sigma^2 > 0$ if and only if it converges in the Wasserstein distance $d_2$ to the same limit.
\end{theorem}

Both assumptions in Theorem \ref{thm:hauptresultat} are rather standard when trying to establish asymptotic normality of a $U$-statistic. In the case of mixing data, one needs to assume $p$-moments for some $p > 2$ instead of just second moments in order to use covariance inequalities for mixing processes. Informally our result may be stated as `If you can prove asymptotic normality, you probably also have convergence in $d_2$.'

The proof of our theorem is not complicated. It makes use of the fact that the asymptotic normality of $U_h$ is determined by the first term in its Hoeffding decomposition, as outlined in the introduction of this article. As this first term is simply a partial sum we then use the fact that partial sums of strictly stationary and mixing data are asymptotically normal if and only if they are uniformly square-integrable (after renormalising with their variances), which implies convergence in $d_2$. Since we approximate $\sqrt{n} U_h$ by $\sqrt{n} U_n^{(1)}(h;\xi)$ not just in probability but in $L_2$, this is enough to show that the $d_2$-convergence of $\sqrt{n} U_n(h;\xi)$ carries over to $\sqrt{n} U_h$.

\begin{proof}[Proof of Theorem \ref{thm:hauptresultat}]
  Convergence in the Wasserstein distance implies weak convergence \citep[Theorem 6.9 in][]{villani:optimal_transport}, so one direction of our claim is trivial. We will therefore assume in both cases that $\sqrt{n} U_h$ converges weakly to some $\mathcal{N}\left(0, \sigma^2\right)$ and prove the convergence in $d_2$.

  Consider the Hoeffding decomposition \eqref{eq:hoeffding_zerlegung}. For any $q \geq 1$, Jensen's inequality implies that
  $$
    \|\textbf{h}_k(X_1, \ldots, X_k)\|_{L_q} \leq \|h(X_1, \ldots, X_m)\|_{L_q}.
  $$
  Furthermore, by the triangle inequality,
  $$
    \|h_i(X_1, \ldots, X_i)\|_{L_q} \leq \sum_{k=0}^i {i \choose k} \|h(X_1, \ldots, X_m)\|_{L_q} \leq 2^m \|h(X_1, \ldots, X_m)\|_{L_q}
  $$
  for any $0 \leq i \leq m$. Thus under the first assumption in the statement of the theorem, each of the Hoeffding-terms $h_i$ has finite second moments; under the second assumption, finite $p$-moments.

  We will show that the limiting behaviour of $\sqrt{n} U_h$ is solely determined by the first-order terms in the Hoeffding decomposition (this is standard practice in limit theorems for $U$-statistics). Note that
  \begin{equation}
    \label{eq:l2_abstand}
    \|\sqrt{n} U_h - m \sqrt{n} U_n^{(1)}(h;\xi)\|_{L_2} = \sqrt{n}\left\|\sum_{i=2}^m {m \choose i} U_n^{(i)}(h;\xi)\right\|_{L_2} \leq \sqrt{n} \sum_{i=2}^m {m \choose i} \left\|U_n^{(i)}(h;\xi)\right\|_{L_2}.
  \end{equation}
  If the first assumption in the statement of the theorem is satisfied, we can use elementary arguments for $U$-statistics of i.i.d. data \citep[e.g.\@ Lemma 5.2.1 A in][]{serfling:approximation_theorems} to see that $\left\|U_n^{(i)}(h;\xi)\right\|_{L_2}$ is of order $\mathcal{O}\left(n^{-i/2}\right)$. On the other hand, under the second assumption in the statement of the theorem, Lemma 3 in \cite{arcones:1998} implies that each of the norms $\|U_n^{(i)}(h;\xi)\|_{L_2}$ is bounded by $c \, n^{-i/2}$, where the constant $c$ only depends on the degree $m$, the rate of the mixing coefficients $\beta(n)$ and the $p$-th moment $\|h(X_1, \ldots, X_m)\|_{L_p}$. Thus, under either assumption, the entire right-hand side in Eq.\@ \eqref{eq:l2_abstand} is $\mathcal{O}\left(1/\sqrt{n}\right)$, and so $m \sqrt{n} U_n^{(1)} = m n^{-1/2}\sum_{k=1}^n h_1(X_k;\xi)$ converges in distribution to the same weak limit $\mathcal{N}\left(0, \sigma^2\right)$ as the original $U$-statistic $\sqrt{n} U_h$.

  Under the first assumption of the theorem, the variance of $(m/\sqrt{n}) \sum_{k=1}^n h_1(X_k;\xi)$ is equal to $m^2 \,\mathrm{Var}(h_1(X_1;\xi))$ for all $n \in \mathbb{N}$, and thus its limit for $n \to \infty$ trivially exists. Under the second assumption, this limit exists by virtue of Theorem 10.7 in \cite{bradley:mixing}. Let us denote this limit by $\tilde{\sigma}^2$, i.e.\@
  $$
    \tilde{\sigma}^2 = \lim_{n \to \infty} \mathrm{Var}\left(\frac{m}{\sqrt{n}} \sum_{k=1}^n h_1(X_k;\xi)\right) = \lim_{n \to \infty} n^{-1} \sigma_n^2,
  $$
  where $\sigma_n^2$ denotes the variance of the non-standardised partial sum $m \sum_{k=1}^n h_1(X_k;\xi)$. Now,
  \begin{equation}
    \label{eq:normierung_varianzen}
    \sigma_n^{-1} m \sum_{k=1}^n h_{1}(X_k;\xi) = \frac{\tilde{\sigma }}{n^{-1/2} \sigma_n} \,\frac{m}{\tilde{\sigma}\sqrt{n}} \sum_{k=1}^n h_1(X_k;\xi) \xrightarrow[n \to \infty]{\mathcal{D}} \mathcal{N}\left(0, (\sigma/\tilde{\sigma})^2\right).
  \end{equation}
  Theorem 10.2 in \cite{bradley:mixing} now implies that the random variables $\sigma_n^{-1}m \sum_{k=1}^n h_1(X_k;\xi)$ are uniformly square-integrable, which -- again by Eq.\@ \eqref{eq:normierung_varianzen} -- implies uniform square-integrability of $m n^{-1/2} \sum_{k=1}^n h_1(X_k;\xi) = m \sqrt{n} U_n^{(1)}(h;\xi)$. By Theorem 6.9 in \cite{villani:optimal_transport}, the uniform square-integrability combined with the already established weak convergence of $m \sqrt{n} U_n^{(1)}(h;\xi)$ implies that $m \sqrt{n} U_n^{(1)}(h;\xi)$ also converges in the Wasserstein distance $d_2$ to its weak limit $\mathcal{N}\left(0, \sigma^2\right)$. This proves our claim since
  \begin{align*}
    d_2\left(\sqrt{n}U_h, \mathcal{N}\left(0, \sigma^2\right)\right)
     & \leq d_2\left(\sqrt{n} U_h, m \sqrt{n}U_n^{(1)}(h;\xi)\right) + d_2\left(m \sqrt{n}U_n^{(1)}(h;\xi), \mathcal{N}\left(0, \sigma^2\right)\right) \\
     & \leq \|\sqrt{n} U_h - m \sqrt{n} U_n^{(1)}(h;\xi)\|_{L_2} + d_2\left(m \sqrt{n}U_n^{(1)}(h;\xi), \mathcal{N}\left(0, \sigma^2\right)\right),
  \end{align*}
  and this upper bound converges to $0$ by Eq.\@ \eqref{eq:l2_abstand} and the observations above.
\end{proof}

\begin{remark*}
  The final part of the proof of Theorem \ref{thm:hauptresultat} uses two results from the theory of mixing processes, namely Theorems 10.2 and 10.7 in \cite{bradley:mixing}. The latter is only needed to cover the second assumption in Theorem \ref{thm:hauptresultat}, and the former is used to show that the asymptotic normality of $m \sqrt{n} U_n^{(1)}(h;\xi)$ implies uniform square-integrability. If we are only interested in the i.i.d. case, i.e.\@ we are working under the first assumption of Theorem \ref{thm:hauptresultat}, we do not need Theorem 10.2 in \cite{bradley:mixing} for this: Under the i.i.d. assumption, Eq.\@ \eqref{eq:normierung_varianzen} and Feller's theorem imply that the sequence $(m \sum_{k=1}^n h_1(X_k;\xi))_{n \in \mathbb{N}}$ fulfills the Lindeberg condition. As pointed out by Billingsley \citep[see Eq.\@ (12.20) in][and the paragraph thereafter]{billingsley:convergence_of_probability_measures}, the Lindeberg condition implies uniform square-integrability. Thus, the proof of Theorem \ref{thm:hauptresultat} under i.i.d. assumptions requires only elementary arguments.
\end{remark*}

\bibliographystyle{abbrvnat}
\bibliography{ustat_bib}
\end{document}